\documentclass[12pt,a4paper]{amsart}
\usepackage{amssymb}
\usepackage{amsthm}
\usepackage{amsmath}
\usepackage{amsxtra}
\usepackage{latexsym}
\usepackage{mathrsfs}
\usepackage[all,cmtip]{xy}
\usepackage[all]{xy}
\usepackage{enumitem}
\usepackage{xcolor}
\usepackage{graphicx}
\usepackage{comment}

\usepackage[dvipdfmx]{hyperref}

\newtheorem{thm}{Theorem}[section]
\newtheorem{lem}[thm]{Lemma}

\newtheorem{prop}[thm]{Proposition}

\theoremstyle{definition}
\newtheorem{defn}[thm]{Definition}
\newtheorem{rmk}[thm]{Remark}
\newtheorem*{ack}{Acknowledgements}
\numberwithin{equation}{section}

\def\Q{{\mathbb Q}}
\def\R{{\mathbb R}}

\def\P{{\mathbb P}}

\def\NE{\overline{NE}}

\DeclareMathOperator{\Pic}{Pic}

\DeclareMathOperator{\id}{id}

\DeclareMathOperator{\Exc}{Exc}
\DeclareMathOperator{\Supp}{Supp}
\DeclareMathOperator{\Sym}{Sym}

\DeclareMathOperator{\Eff}{Eff}

\DeclareMathOperator{\dv}{div}

\usepackage{enumitem}
\usepackage{blindtext}

\newenvironment{notation}[0]{%
  \begin{list}%
    {}%
    {
    \setlength{\itemindent}{0pt}
     \setlength{\labelwidth}{4\parindent}
     \setlength{\labelsep}{\parindent}
     \setlength{\leftmargin}{5\parindent}
     \setlength{\itemsep}{0pt}
     }%
   }%
  {\end{list}}

\title[]
{Int-amplified endomorphisms on normal projective surfaces}
\author{Yohsuke Matsuzawa}
\author{Shou Yoshikawa}
\address{Graduate school of Mathematical Sciences, the University of Tokyo, Komaba, Tokyo,
153-8914, Japan}
\address{Graduate school of Mathematical Sciences, the University of Tokyo, Komaba, Tokyo,
153-8914, Japan}
\email{myohsuke@ms.u-tokyo.ac.jp}
\email{yoshikaw@ms.u-tokyo.ac.jp}

\begin{document}

\begin{abstract}
We investigate int-amplified endomorphisms on normal projective surfaces.
We prove that the output of the equivariant MMP is either 
a Q-abelian surface,
a (equivariant) quasi-\'etale quotient of a smooth projective surface,
a Mori dream space, or
a projective cone of an elliptic curve.

\end{abstract}

\maketitle

\setcounter{tocdepth}{1}
\tableofcontents

\section{Introduction} 

In this paper, we work over an algebraically closed field $k$ of characteristic zero.
A self-morphism $f \colon X \longrightarrow X$ on a projective variety $X$ is called int-amplified
if there exists an ample Cartier divisor $H$ on $X$ such that $f^{*}H-H$ is ample.
Int-amplified endomorphisms are compatible with minimal model program (MMP), as shown in \cite{meng, meng-zhang2}.
Also, existence of such endomorphisms imposes strong  constraint to the singularities of the varieties.
Therefore, it seems possible to classify all int-amplified endomorphisms or varieties admitting an int-amplified endomorphism.
In this paper, we investigate int-amplified endomorphisms on normal projective surfaces.

To state our main theorem, we fix the terminology.
\begin{defn}\

\begin{enumerate}
\item A morphism $h \colon Y \longrightarrow X$ between varieties is called quasi-\'etale if $h$ is \'etale at every codimension one point on $Y$.
\item A variety $X$ is called Q-abelian if there exists a finite surjective quasi-\'etale morphism $A \longrightarrow X$ from an abelian variety $A$.
\end{enumerate}
\end{defn}
\noindent
The linear equivalence and $\Q$-linear equivalence of divisors on normal projective varieties are denoted by $\sim$ and $\sim_{\Q}$ respectively.
The Iitaka dimension of a $\Q$-Cartier divisor $D$ on a normal projective variety is denoted by $\kappa(D)$.

The following is the main theorem of this paper.

\begin{thm}\label{thm: main}
Let $X$ be a normal projective surface over $k$.
Let $f \colon X \longrightarrow X$ be an int-amplified endomorphism.
Then $X$ is $\Q$-Gorenstein log canonical (lc) and we have the following sequence of morphisms:
\begin{align*}
X=X_{1} \longrightarrow \cdots \longrightarrow X_{r} \longrightarrow C
\end{align*}
where
\begin{itemize}
\item $X_{i} \longrightarrow X_{i+1}$ is the divisorial contraction of a $K_{X_{i}}$-negative extremal ray for $i=1,\dots ,r-1$;
\item $X_{r} \longrightarrow C$ is a Fano contraction of a $K_{X_{r}}$-negative extremal ray if $K_{X}$ is not pseudo-effective;
\item we ignore `` $ \longrightarrow C$ '' if $K_{X}$ is pseudo-effective;
\item there exists a positive integer $n$ such that $f^{n}$ induces endomorphisms on $X_{i}$ and $C$ (in such case, we call the sequence $f^{n}$-equivariant MMP).
\end{itemize}
Moreover, one of the following holds:

\begin{enumerate}
\item\label{Qabcase} $K_{X_{1}}\sim_{\Q} 0$. In this case $r=1$. 
If $X$ is Kawamata log terminal (klt), then $X$ is a Q-abelian variety;


\item $C$ is an elliptic curve, $r=1$ and $X_{1}$ is smooth;

\item\label{k=0case} $C\simeq \P^{1}$, $\kappa(-K_{X_{r}})=0$. In this case, $X$ is klt, $r=1$, there exists a quasi-\'etale finite surjection $h \colon Y \longrightarrow X$
of degree $2$ from a smooth projective surface $Y$, which is a minimal ruled surface over an elliptic curve,  and an endomorphism $f_{Y} \colon Y \longrightarrow Y$ such that
\[
\xymatrix{
Y \ar[r]^{f_{Y}} \ar[d]_{h} & Y \ar[d]^{h} \\
X \ar[r]_{f^{n}} & X
}
\]
is commutative;

\item\label{k=1case} $C\simeq \P^{1}$, $\kappa(-K_{X_{r}})=1$.  In this case, $X$ is klt, $r=1$,  there exists a quasi-\'etale finite surjection $h \colon Y \longrightarrow X$
from a smooth projective surface $Y$, which is a minimal ruled surface over an elliptic curve,   and an endomorphism $f_{Y} \colon Y \longrightarrow Y$ such that
\[
\xymatrix{
Y \ar[r]^{f_{Y}} \ar[d]_{h} & Y \ar[d]^{h} \\
X \ar[r]_{f^{n}} & X
}
\]
is commutative;

\item\label{k=2case} $C\simeq \P^{1}$, $\kappa(-K_{X_{r}})=2$. In this case, $X$ is klt and a Mori dream space.

\item\label{ptcase} $C$ is a point, the Picard number of $X_{r}$ is one and $-K_{X_{r}}$ is ample.
In this case, $X$ is a projective cone of an elliptic curve or a Mori dream space.
\end{enumerate}

\end{thm}

\begin{rmk}
The structure of $X$ in the cases (\ref{Qabcase}) and (\ref{ptcase}) are already known (cf. \cite{meng, brou-gon}).
The essential result of this paper is the construction of quasi-\'etale covers in the cases (\ref{k=0case}) and (\ref{k=1case}).
\end{rmk}

\begin{rmk}
We refer \cite[Definiton 1.10]{huke} for the definition of Mori dream spaces.
\end{rmk}

\begin{rmk}
Notation as in Theorem \ref{thm: main}.
Let $g \colon X \longrightarrow X$ be any surjective endomorphism.
Then, by \cite[Theorem 4.6]{meng-zhang2}, $g^{m}$ induces endomorphisms on $X_{r}$ and $C$ for some $m>0$.
Moreover, in case (\ref{k=0case}), the induced endomorphism $g_{r}$ on $X_{r}$ lifts to an endomorphism on $Y$. 
Indeed, by the proof of Lemma \ref{keylem}, the curve `` $C$ '' in Lemma \ref{keylem} and Proposition \ref{qetale-cover} is totally invariant under $g_{r}$.
Therefore, by  Proposition \ref{qetale-cover}, $g_{r}$ lifts to the quasi-\'etale cover.
\end{rmk}

\section{Notation and Terminology} 
Throughout this paper, the ground field $k$ is an algebraically closed field of characteristic zero.
A variety is an irreducible reduced separated scheme of finite type over $k$.
A subvariety means an irreducible reduced closed subscheme. 
Divisor on a normal projective variety means Weil divisor.

For a self-morphism $f \colon X \longrightarrow X$ of a variety $X$, a subset $S \subset X$ is called totally invariant under $f$ if $f^{-1}(S)=S$ as sets.

\begin{notation}
\item[$ \overline{\Eff}$] The pseudo-effective cone of a projective variety $X$ is denoted by $ \overline{\Eff}(X)$.
\item[$R_{f}$] The ramification divisor of a finite surjective morphism $f \colon X \longrightarrow Y$ between normal projective varieties is denoted by $R_{f}$.
\item[$D \geq E$] Let $D, E$ be two $\Q$-Weil divisors on a normal projective variety. We write $D\geq E$ if the divisor $D-E$ is effective.
\end{notation}

\section{Preliminaries} 

\subsection{}

Let $X$ be a normal variety, and let $\mu \colon X' \longrightarrow X$ be a proper birational morphism from 
a normal variety $X'$. If $\Delta \subset X$ is a $\Q$-divisor, we denote by 
$\mu_*^{-1}(\Delta)$ its strict transform. 

A log pair is a tuple $(X, \Delta)$ where $X$ is a normal variety
and $\Delta=\sum_i d_i \Delta_i$ is a $\Q$-divisor on $X$ with $d_i \leq 1$ for all $i$. 
We say that 
the pair $(X, \Delta)$ is log canonical (lc) (resp. purely log terminal (plt), resp. Kawamata log terminal (klt)) if 
$K_X+\Delta$ is $\Q$-Cartier and
for every proper birational morphism 
$\mu \colon X' \longrightarrow X$ from a normal variety $X'$
we can write  
$$
K_{X'}+\mu_*^{-1}(\Delta) = \mu^* (K_X+\Delta) + \sum_j a(E_j, X, \Delta) E_j,
$$
where the divisor $E_j$ are $\mu$-exceptional and $a(E_j, X, \Delta) \geq -1$ (resp. $a(E_j, X, \Delta)>-1$, resp. $a(E_j, X, \Delta)>-1$ and $d_{i}<1$ for all $i$) for all $j$.
If the pair $(X, \Delta)$ is lc, we say that a subvariety $Z \subset X$ is an lc center if there exists
a morphism $\mu \colon X' \longrightarrow X$ as above and a $\mu$-exceptional divisor $E$ such that $Z=\mu(E)$ 
and $a(E, X, \Delta)=-1$.

A variety $X$ is called lc, (resp. klt) if so is the pair $(X, 0)$.
A variety $X$ is called $\Q$-Gorenstein if the canonical divisor $K_{X}$ is $\Q$-Cartier and $\Q$-factorial if every Weil divisor on $X$ is $\Q$-Cartier.
If a variety is lc, then it is $\Q$-Gorenstein by definiton.
A surface is $\Q$-factorial if it has rational singularities and it has rational singularities if it is klt (\cite[Theorem 5.22]{kolmor}, \cite[Theorem 4.6]{bades}).

\subsection{}

We gather several facts on endomorphisms that we use later.
The first two lemmas are about the relationship between endomorphisms and singularities.

\begin{lem}[\textup{\cite[Proposition 7.7]{yoshi}, cf. \cite[Lemma 2.10, Theorem 1.4]{brou-hor}}]\label{normality-of-inv-curve}
Let $X$ be a normal projective surface and $f \colon X \longrightarrow X$ a surjective endomorphism with $\deg f>1$.
Let $C \subset X$ be a reduced effective divisor such that $f^{-1}(C)=C$.
Then $(X,C)$ is a lc $\Q$-Gorenstein pair and any lc center of $(X,C)$ is not contained in $\Supp{R_f}$ and totally invariant if we replace $f$ by a suitable power $f^n$
\end{lem}

By setting $C=0$, we get the following.

\begin{lem}\label{lem:endom-and-sing}
Let $X$ be normal projective surface and $f \colon X \longrightarrow X$ a surjective endomorphism with $\deg f>1$.
Then $X$ is $\Q$-Gorenstein lc and any lc center of $X$ is not contained in $\Supp{R_f}$ and totally invariant if we replace $f$ by a suitable power $f^n$.
\end{lem}

We recall basic properties and fundamental theorems on int-amplified endomorphisms.

\begin{lem}\label{property-of-intamp}\  
\begin{enumerate}
\item Let $X$ be a normal projective variety, $f \colon X \longrightarrow X$ a surjective morphism, and $n>0$ a positive integer.
Then $f$ is int-amplified if and only if so is $f^{n}$.

\item Let $\pi \colon X \longrightarrow Y$ be a surjective morphism between normal projective varieties.
Let $f \colon X \longrightarrow X$, $g \colon Y \longrightarrow Y$ be surjective endomorphisms such that $\pi \circ f=g \circ \pi$.
If $f$ is int-amplified, then so is $g$. 

\item  Let $\pi \colon X \dashrightarrow Y$ be a dominant rational map between normal projective varieties of same dimension.
Let $f \colon X \longrightarrow X$, $g \colon Y \longrightarrow Y$ be surjective endomorphisms such that $\pi \circ f=g \circ \pi$.
Then $f$ is int-amplified if and only if so is $g$.
\end{enumerate}
\end{lem}
\begin{proof}
See \cite[Lemma 3.3, 3,5, 3.6]{meng}.
\end{proof}

\begin{lem}[\textup{\cite[Theorem 1.5]{meng}}]\label{lem:eff-anti-canonical}
Let $X$ be a normal $\Q$-Gorenstein projective variety and $f \colon X \longrightarrow X$ an int-amplified endomorphism.
Then $-K_X$ is numerically equivalent to an effective $\Q$-Cartier divisor.
\end{lem}

\begin{prop}[\textup{\cite[Theorem 5.2]{meng}}]
Let $X$ be a normal $\Q$-Gorenstein projective variety and $f \colon X \longrightarrow X$ an int-amplified endomorphism.
If $K_{X}$ is pseudo-effective, then $K_{X}\sim_{\Q}0$.
If, moreover, $X$ is klt, then $X$ is a Q-abelian variety, there exists a quasi-\'etale finite morphism $A \longrightarrow X$
from an abelian variety $A$ and some power $f^{n}$ of $f$ lifts to a self-morphism of $A$.
\end{prop}

The following easy lemma makes MMP equivariant under certain endomorphisms.

\begin{lem}\label{equiv-contr}
Let $X$ be a lc projective variety and $f \colon X \longrightarrow X$ a surjective endomorphism.
Let $R \subset \NE(X)$ be a $K_{X}$-negative extremal ray and $\pi \colon X \longrightarrow Y$ the contraction of $R$.
Suppose $f_{*}R=R$. Then there exists a surjective endomorphism $Y \longrightarrow Y$ such that
\[
\xymatrix{
X \ar[r]^{f} \ar[d]_{\pi} & X \ar[d]^{\pi}\\
Y \ar[r] & Y.
}
\]
\end{lem}
\begin{proof}
This is true because the contraction is determined by the ray.
\end{proof}

We will use the following lemma to prove kltness.

\begin{lem}\label{lem:fiber-klt}
Consider the following commutative diagram
\[
\xymatrix{
X \ar[r]^{f} \ar[d]_{\pi} & X \ar[d]^{\pi}\\
C \ar[r]_{g}& C
}
\]
where $X$ is a normal projective surface, $C$ is a smooth projective curve, $f$ is an int-amplified endomorphism, $g$ is an endomorphism and 
$\pi$ is a surjective morphism with connected fibers.
Then $X$ is klt.
\end{lem}

\begin{proof}
By Lemma \ref{lem:endom-and-sing}, $X$ is $\Q$-Gorenstein lc and we may assume a lc center $P$ of $X$ is totally invariant under $f$.
Then $\pi(P)$ is totally invarinat and the fibre $F$ of $P$ is also totally invariant.
In particular, since $F_{red} \leq R_f$, we have $P \in \mathrm{Supp}(R_f)$, but this contradicts to Lemma \ref{lem:endom-and-sing}.
\end{proof}

\section{Int-amplified endomorphisms on two dimensional Mori fiber spaces}

\begin{prop}\label{end-mfs-sing}
Consider the following commutative diagram
\[
\xymatrix{
X \ar[r]^{f} \ar[d]_{\pi} & X \ar[d]^{\pi}\\
C \ar[r]_{g}& C
}
\]
where $X$ is a $\Q$-Gorenstein lc projective surface, $f$ is an int-amplified endomorphism,
$\pi$ is a $K_{X}$-negative extremal ray contraction to a smooth projective curve $C$ and $g$ is an endomorphism.
Then

\begin{enumerate}
\item $C$ is isomorphic to $\P^{1}$ or an elliptic curve;
\item \label{elliptic base} If $C$ is an elliptic curve, then $f$ does not have non-empty totally invariant finite set and $X$ is smooth;
\item If $C\simeq \P^{1}$ and $-K_{X}$ is not big, then $f$ does not have non-empty totally invariant finite set.
\end{enumerate}

\end{prop}
\begin{proof}
(1) Since $f$ is int-amplified, so is $g$ and that means $\deg g>1$.
This implies $C$ is isomorphic to $\P^{1}$ or an elliptic curve.

(2) Suppose $C$ is an elliptic curve.
Then $g$ is an \'etale non-isomorphic morphism, and therefore $g$ and its iterates have no totally invariant points.
Thus $f$ also does not have non-empty totally invariant finite set. By Lemma \ref{lem:endom-and-sing}, $X$ is klt and $\Q$-factorial.

If $\pi$ has a singular fiber, it is not generically reduced.
Indeed, if a fiber $F$ of $\pi$ is generically reduced, it is integral.
(It is irreducible because $\pi$ is a Mori fiber space over a curve and $X$ is $\Q$-factorial. Every fiber of $\pi$ is Cohen-Macaulay and
thus it is reduced if generically reduced.)
Since $\pi$ is flat and general fibers are $\P^{1}$, the arithmetic genus of $F$ is zero and it implies $F\simeq \P^{1}$.
This is a contradiction.
 
Assume $\pi$ has a singular fiber $F=\pi^{*}P$.
Since $g$ is \'etale,  $(g^{n})^{*}P$ is a reduced divisor but every coefficient of $\pi^{*}(g^{n})^{*}P=(f^{n})^{*}F$ is greater than one for any $n$.
This implies there are infinitely many singular fibers of $\pi$, but this is absurd.
Thus all fibers of $\pi$ are regular and therefore $X$ is smooth.

(3)
Note that by Lemma \ref{lem:fiber-klt}, $X$ is klt and $\Q$-factorial. 
Assume $f$ admits a totally invariant finite set.
Replacing $f$ by its iterate, we may assume $f$ has a totally invariant point.
Since $-K_{X}$ is not big and the Picard number of $X$ is two, $-K_{X}$ generates an extremal ray of
the pseudo-effective cone $ \overline{\Eff}(X)$.
Another ray is generated by the fiber class $F$.
Since $F$ is preserved under $f^{*}$, $-K_{X}$ is also preserved and we write $f^{*}(-K_{X}) \equiv q (-K_{X})$ where $q$ is an integer greater than one.
Then $R_{f}\equiv K_{X}-f^{*}K_{X}\equiv (q-1)(-K_{X})$, i.e. $R_{f}$ generates the extremal ray different than the one generated by $F$.
Now the reduced fiber containing the totally invariant point is contained in the support of $R_{f}$.
This is a contradiction.
\end{proof}

\begin{lem}\label{keylem}
Consider the following commutative diagram
\[
\xymatrix{
X \ar[r]^{f} \ar[d]_{\pi} & X \ar[d]^{\pi}\\
\P^{1} \ar[r]_{g}& \P^{1}
}
\]
where $X$ is a klt projective surface, $f$ is an int-amplified endomorphism,
$\pi$ is a $K_{X}$-negative extremal ray contraction and $g$ is an endomorphism.
Let $R_{f}$ be the ramification divisor of $f$.
If $ \kappa(-K_{X})=0$, then $f^{*}(-K_{X})\sim_{\Q}q(-K_{X})$ for some integer $q>1$,
$(R_{f})_{\rm red}=:C$ is a smooth irreducible curve and the following holds:
\begin{itemize}
\item $C\sim_{\Q}-K_{X}$;
\item $f^{-1}(C)=C$ as sets; 
\item $R_{f}=(q-1)C$ as Weil divisors.
\end{itemize}
\end{lem}

\begin{proof}
Note that $X$ is $\Q$-factorial since it is a klt surface.
Moreover, $\Pic(X)_{\Q}\simeq N^{1}(X)_{\Q}$ since $X$ is rational.
Let $ \overline{\Eff}(X)= \overline{NE}(X)=\R_{\geq0}F+\R_{\geq0}v$ where $F$ is the fiber class.
Note that we have
\[
f^{*}F=\deg g F,\ f^{*}v=qv
\]
for some $q\in \R_{>1}$.
By Lemma \ref{lem:eff-anti-canonical}, we can write $-K_{X}=aF+bv$ in $N^{1}(X)_{\R}$ for some $a, b\geq0$.
Since $\pi$ is a $K_{X}$-negative contraction, we have $0<(-K_{X}\cdot F)=b(v\cdot F)$.
This implies $b>0$ and $(v\cdot F)>0$.
Therefore, $a=0$. Indeed, if $a>0$, $-K_{X}$ is contained in the interior of $ \overline{\Eff}(X)$ and it means $-K_{X}$ is big, which
contradicts to our assumption.
Thus $-K_{X}$ generates an extremal ray, $q$ is an integer, and $f^{*}(-K_{X}) \sim_{\Q} q(-K_{X})$.

Now, since $R_{f}\sim K_{X}-f^{*}K_{X} \sim_{\Q} (q-1)(-K_{X})$, $\kappa(R_{f})=0$ and $R_{f}$ generate the extremal ray of $ \overline{\Eff}(X)$.
This implies $R_{f}$ is irreducible.
Set $C=(R_{f})_{\rm red}$. 
Since $f^{*}R_{f} \sim_{\Q} q R_{f}$ and $\kappa(R_{f})=0$, $f^{-1}(R_{f})=R_{f}$ (in other words, $f^{-1}(C)=C$) as sets.
Thus, by the definition of the ramification divisor, $R_{f}=(q-1)C$.
From this, we get $-K_{X}\sim_{\Q}C$.

Now we apply Lemma \ref{normality-of-inv-curve}.
Since $f$ does not ramify along fibers, there is no totally invariant finite set.
Thus, by Lemma \ref{normality-of-inv-curve}, $(X,C)$ has no lc center.                 
Then $(X,C)$ is plt, in particular, $C$ is normal (cf. \cite[Proposition 5.51]{kolmor}).

\end{proof}

\begin{prop}\label{qetale-cover}
Consider the following commutative diagram
\[
\xymatrix{
X \ar[r]^{f} \ar[d]_{\pi} & X \ar[d]^{\pi}\\
\P^{1} \ar[r]_{g}& \P^{1}
}
\]
where $X$ is a klt projective surface, $f$ is an int-amplified endomorphism,
$\pi$ is a $K_{X}$-negative extremal ray contraction and $g$ is an endomorphism.
Let $R_{f}$ be the ramification divisor of $f$.
If $ \kappa(-K_{X})=0$,
then $(R_{f})_{\rm red}=:C$ is an elliptic curve.
Moreover, let $X'=X \times_{\P^{1}}C$ and $ \widetilde{X}$ be the normalization of $X'_{\rm red}$.
Then 
\begin{itemize}
\item $ \widetilde{ X}$ is smooth;
\item the projection $ \widetilde{\pi} \colon \widetilde{X} \longrightarrow C$ is a Fano contraction of a $K_{ \widetilde{X}}$ negative extremal ray (i.e. $ \widetilde{X}$ is a minimal ruled surface over $C$);
\item the finite morphism $h \colon \widetilde{X} \longrightarrow X$ is quasi-\'etale of degree $2$;
\item  $f$ induces an int-amplified endomorphism on $ \widetilde{X}$:
\end{itemize}
\[
\xymatrix{
\widetilde{X} \ar[r] \ar[rd]_{ \widetilde{\pi}} \ar@/^1pc/[rr]^{h} & X' \ar[d] \ar[r] & X \ar[d]^{\pi}\\
& C \ar[r] & \P^{1}.
}
\]
\end{prop}
\begin{proof}
We use the notation in Lemma \ref{keylem}.
The restriction of $f$ on $C$ has degree lager than one, so $C$ is isomorphic to $\P^{1}$ or an elliptic curve.
Note that $\pi|_{C} \colon C \longrightarrow \P^{1}$ is a double cover.

{\bf Step 1}.
We assume $C\simeq \P^{1}$ and deduce contradiction.
Form the following commutative diagram:
\[
\xymatrix{
\widetilde{X} \ar[r] \ar[rd]_{ \widetilde{\pi}} \ar@/^1pc/[rr]^{h} & X' \ar[d] \ar[r] & X \ar[d]^{\pi}\\
& C \ar[r]_{\pi|_{C}} & \P^{1}
}
\]
where $X'=X \times_{C}\P^{1}$ and $ \widetilde{X}$ is the normalization of $(X')_{\rm red}$.
Since $f$ induces an endomorphism of $C$, it induces an endomorphism $ \widetilde{f}$ on $\widetilde{X}$ which is 
int-amplified.
By Lemma \ref{lem:fiber-klt}, $ \widetilde{X}$ is klt.
By Lemma \ref{lem:eff-anti-canonical}, $-K_{ \widetilde{X}}$ is $\Q$-linearly effective 
(note that $ \widetilde{X}$ is rational since general fibers of $ \widetilde{\pi}$ is rational).
Let $R_{h}$ be the ramification divisor of $h$.
Then, by pushing the ramification formula by $h$, we get
\[
(\deg h)(-K_{X}) \sim h_{*}(-K_{ \widetilde{X}})+h_{*}R_{h}.
\]
Since $\kappa(-K_{X})=0$ and $-K_{X}\sim_{\Q} C$, we get $\Supp R_{h} \subset h^{-1}(C)$.
By the construction of $h$, $h$ is not ramified along horizontal divisors.
Thus $R_{h}=0$ and $h$ is quasi-\'etale.
Then we get $R_{ \widetilde{f}}=h^{*}R_{f}$ where $R_{ \widetilde{f}}$ is the ramification divisor of $ \widetilde{f}$.
In particular, $ \widetilde{f}$ does not ramify along curves contracted by $ \widetilde{\pi}$.
This implies $ \widetilde{f}$ does not have totally invariant finite set.
Moreover, if there is a $K_{ \widetilde{X}}$-negative extremal divisorial contraction, 
it is equivariant with respect to some iterate of $ \widetilde{f}$ by \cite[Theorem 4.6]{meng-zhang2}.
Then it produces a totally invariant point of some iterate of $f|_{C}$, which is absurd because $R_{ \widetilde{f}}$ is horizontal.
Therefore, $ \widetilde{\pi}$ is a Fano contraction (Note $K_{ \widetilde{X}} \sim h^{*}K_{X}$ is not nef over $C$).
Since $h_{*}(-K_{ \widetilde{X}}) \sim \deg h (-K_{X})$ and $\kappa(-K_{X})=0$, $\kappa(-K_{ \widetilde{X}})=0$.
Now, we can apply Lemma \ref{keylem} to $ \widetilde{X}$ and $ \widetilde{f}$, and it says $R_{ \widetilde{f}}$ is irreducible.
But $\Supp R_{h}=h^{-1}(C)$ is not irreducible since $\pi|_{C} \colon C \longrightarrow \P^{1}$ has degree two.
This is contradiction.

{\bf Step 2}.
Now we assume $C$ is an elliptic curve.
Form the following commutative diagram as in Step 1:
\[
\xymatrix{
\widetilde{X} \ar[r] \ar[rd]_{ \widetilde{\pi}} \ar@/^1pc/[rr]^{h} & X' \ar[d] \ar[r] & X \ar[d]^{\pi}\\
& C \ar[r]_{\pi|_{C}} & \P^{1}.
}
\]
Since $\pi|_{C}$ is a double cover, $h$ has degree $2$.
As in Step 1, $f$ induces an int-amplified endomorphism $ \widetilde{f}$ on $ \widetilde{X}$ and 
$ \widetilde{X }$ is $\Q$-Gorenstein lc.
Consider the following equations:
\begin{align}
&R_{h}+h^{*}R_{f}=R_{ \widetilde{f}}+ \widetilde{f}^{*}R_{h}; \label{eq:1}\\
&h^{*}R_{f}=(q-1)h^{*}C; \label{eq:2}\\
&{ \widetilde{f}}^{*}h^{*}C=h^{*} f^{*}C=qh^{*}C. \label{eq:3}
\end{align}
By construction, $h^{*}C$ has two components and each coefficient is $1$.
By (\ref{eq:3}), $R_{ \widetilde{f}}-(q-1)h^{*}C\geq0$.
By (\ref{eq:2}), $R_{ \widetilde{f}}-h^{*}R_{f}\geq0$.
By (\ref{eq:1}), $R_{h}- \widetilde{f}^{*}R_{h}\geq 0$, and this implies $R_{h}$ is totally invariant under $f$ as a set.
Since $h$ is not ramified along horizontal curves by construction, every component of $R_{h}$ is contracted by $ \widetilde{\pi}$.
If $R_{h}\neq 0$, $f|_{C}$ has a non-empty totally invariant set.
This is absurd because $f|_{C}$ is \'etale and not isomorphic.
Therefore, we get $R_{h}=0$, i.e. $h$ is quasi-\'etale.
Moreover, if there is a $K_{ \widetilde{X}}$-negative extremal divisorial contraction, 
it is equivariant with respect to some iterate of $ \widetilde{f}$ by \cite[Theorem 4.6]{meng-zhang2}.
This implies there is a totally invariant point of some iterate of $f|_{C}$, but this is absurd.
Thus $ \widetilde{\pi}$ is a Fano contraction.
By Proposition \ref{end-mfs-sing}(2), $ \widetilde{X}$ is smooth.

\end{proof}

\begin{prop}\label{qetale-cover2}
Consider the following commutative diagram
\[
\xymatrix{
X \ar[r]^{f} \ar[d]_{\pi} & X \ar[d]^{\pi}\\
\P^{1} \ar[r]_{g}& \P^{1}
}
\]
where $X$ is a klt projective surface with $\kappa(-K_{X})=1$, $f$ is an int-amplified endomorphism,
$\pi$ is a $K_{X}$-negative extremal ray contraction and $g$ is an endomorphism.
Then there exists a positive integer $n$ and an elliptic curve $E$ on $X$ such that $f^{n}(E)=E$ satisfying the following properties.
Let $X'=X \times_{\P^{1}}E$ and $ \widetilde{X}$ be the normalization of $X'_{\rm red}$.
Then 
\begin{itemize}
\item $ \widetilde{ X}$ is smooth;
\item the projection $ \widetilde{\pi} \colon \widetilde{X} \longrightarrow E$ is a Fano contraction of a $K_{ \widetilde{X}}$ negative extremal ray (i.e. $ \widetilde{X}$ is a minimal ruled surface over $E$);
\item the finite morphism $h \colon \widetilde{X} \longrightarrow X$ is quasi-\'etale;
\item  $f^{n}$ induces an int-amplified endomorphism on $ \widetilde{X}$:
\end{itemize}
\[
\xymatrix{
\widetilde{X} \ar[r] \ar[rd]_{ \widetilde{\pi}} \ar@/^1pc/[rr]^{h} & X' \ar[d] \ar[r] & X \ar[d]^{\pi}\\
& E \ar[r]_{\pi|_{E}} & \P^{1}.
}
\]
\end{prop}
\begin{proof}
Since $-K_{X}$ is not big, $-K_{X}$ generates the extremal ray of $ \overline{\Eff}(X)$ other than the one generated by the fiber class of $\pi$ 
(cf. the proof of Lemma \ref{keylem}).
Therefore, we can show $(-K_{X})^{2}\geq0$. 
Since the other extremal ray is $K_{X}$-negative,  we get $(-K_{X})^{2}=0$ (otherwise, $-K_{X}$ is ample, but $\kappa(-K_{X})=1$).
Moreover, $-K_{X}$ is semi-ample because it is $\Q$-linearly equivalent to at least two irreducible effective divisors and has self-intersection $0$.
Let $\mu \colon X \longrightarrow \P^{1}$ be the morphism defined by $-mK_{X}$ for sufficiently divisible $m$.
Since $f$ preserves the ray $\R_{\geq0}(-K_{X})$, it induces a non-inveritble endomorphism $g' \colon \P^{1} \longrightarrow \P^{1}$ such that
\[
\xymatrix{
X \ar[r]^{f} \ar[d]_{\mu}& X \ar[d]^{\mu}\\
\P^{1} \ar[r]_{g'} & \P^{1}
}
\]
is commutative.

Since $g'$ is non-isomorphic, it has infinitely many periodic points (cf. \cite{fak}).
General fibers of $\mu$ are elliptic curves because $(K_{X})^{2}=0$.
Thus, if we replace $f$ by a suitable power, we may assume there exists a point $P\in \P^{1}$ such that
$g'(P)=P$ and $\mu^{-1}(P)=:E$ is an elliptic curve.

Consider the following diagram:
\[
\xymatrix{
\widetilde{X} \ar[r] \ar[rd]_{ \widetilde{\pi}} \ar@/^1pc/[rr]^{h} & X' \ar[d] \ar[r] & X \ar[d]^{\pi}\\
& E \ar[r]_{\pi|_{E}} & \P^{1}.
}
\]
where $X'=X \times_{\P^{1}}E$ and $ \widetilde{X}$ is the normalization of $(X')_{\rm red}$.
Since $E$ is preserved by $f$, it induces an int-amplified endomorphism $ \widetilde{f}$ on $ \widetilde{X}$.
Therefore, $ \widetilde{X}$ is $\Q$-Gorenstein klt by Lemma \ref{lem:fiber-klt}.

First, we prove $h$ is quasi-\'etale.
Let $R_{h}$ be the ramification divisor and fix a canonical divisor $K_{ \widetilde{X}}$ of $ \widetilde{X}$ so that 
$-h^{*}K_{X}=-K_{ \widetilde{X}}+R_{h}$.
By Lemma \ref{lem:eff-anti-canonical}, there exists an effective $\Q$-Cartier divisor $D$ on $ \widetilde{X}$ such that $D\equiv -K_{ \widetilde{X}}$.
Then we get
$-(\deg h)K_{X}\equiv h_{*}R_{h}+h_{*}D$.
For any fiber  $E'$ of $\mu$, we have 
$0=(-(\deg h)K_{X}\cdot E')=(h_{*}R_{h}\cdot E')+(h_{*}D\cdot E')$.
Since $E'$ is nef and $h_{*}R_{h}$, $h_{*}D$ are effective, we get $(h_{*}R_{h}\cdot E')=0$.
Therefore, $h_{*}R_{h}$ has no irreducible component that is contained in a fiber of $\pi$.
Since $h$ is finite, $R_{h}$ also has no  irreducible component that is contained in a fiber of $ \widetilde{\pi}$.
By the construction of $h$, $R_{h}$ has no horizontal component, and hence we get $R_{h}=0$.

By the same argument as in the last part of the proof of Proposition \ref{qetale-cover},
$ \widetilde{\pi}$ is a Fano contraction and $ \widetilde{X}$ is smooth.
\end{proof}

\section{Int-amplified endomorphisms on surfaces with big anti-canonical divisor}

\begin{lem}[\textup{cf.\cite[Theorem 5.5]{brou-gon}}]\label{lem:anti-big-mds}
Let $X$ be a normal $\Q$-factorial rational projective surface with $-K_X$ is big.
Then $X$ is a Mori dream space.
\end{lem}
\begin{proof}
Take the minimal resolution $\nu \colon Y \longrightarrow X$.
Then, by negativity lemma, we have $-K_{Y}=-\nu^{*}K_{X}+E$ where $E$ is a $\nu$-exceptional effective divisor.
In particular, $-K_{Y}$ is also big.
Since $Y$ is rational, $Y$ is a Mori dream space by \cite[Theorem 1]{tvv}.
By \cite[Theorem 1.1]{okaw}, $X$ is also a Mori dream space.
\end{proof}

\begin{lem}\label{lem:mmp-anti-big}
Let $X$ be a normal projective surface.
Let $f \colon X \longrightarrow X$ be an int-amplified endomorphism.
Suppose we have the following $f$-equivalent MMP:
\begin{align*}
X=X_{1} \longrightarrow \cdots \longrightarrow X_{r} \longrightarrow C 
\end{align*}
where
\begin{itemize}
\item $X_{i} \longrightarrow X_{i+1}$ is the divisorial contraction of a $K_{X_{i}}$-negative extremal ray for $i=1,\dots ,r-1$;
\item $X_{r} \longrightarrow C$ is identity or the Fano contraction of a $K_{X_{r}}$-negative extremal ray.
\end{itemize}
If $-K_{X_r}$ is big and $X_r$ is $\Q$-factorial, then $-K_X$ is also big.
\end{lem}

\begin{proof}
Let $\nu \colon X \longrightarrow X_r$ be the composite of the divisorial contractions,
then the all exceptional divisors $E_1 , \dots E_{r-1}$ of $\nu$ are totally invariant and $E_i \leq R_f$ for all $i$ since $f$ is int-amplified (cf. \cite[Lemma 3.12]{meng}). 
Write $-K_{X}\sim_{\Q} \nu^{*}(-K_{X_{r}})+E$ where $E=\sum_{i=1}^{r-1}a_{i}E_{i}$. 
By the ramification formula, we get $(f^{n})^{*}(-K_{X})\sim -K_{X}+(f^{n-1})^{*}R_{f}+\cdots +R_{f}$ for $n>0$.
Since $E_{i}$ are components of $R_{f}$ and totally invariant under $f$,  $E+(f^{n-1})^{*}R_{f}+\cdots +R_{f}$ is effective for large $n$.
Therefore, the divisor
\begin{align*}
(f^{n})^{*}(-K_{X})\sim_{\Q} \nu^{*}(-K_{X_{r}})+E+(f^{n-1})^{*}R_{f}+\cdots +R_{f}
\end{align*}
is big and hence so is $-K_{X}$.

\end{proof}

\begin{prop}[\textup{cf.\cite[Theorem 5.1]{brou-gon}}]\label{mmp-mds} 
Let $X$ be a normal projective surface.
Let $f \colon X \longrightarrow X$ be an int-amplified endomorphism.
Suppose we can run $f$-equivalent MMP:
\begin{align*}
X=X_{1} \longrightarrow \cdots \longrightarrow X_{r} \longrightarrow C 
\end{align*}
where
\begin{itemize}
\item $X_{i} \longrightarrow X_{i+1}$ is the divisorial contraction of a $K_{X_{i}}$-negative extremal ray for $i=1,\dots ,r-1$;
\item $X_{r} \longrightarrow C$ is the Fano contraction of a $K_{X_{r}}$-negative extremal ray;
\item $C$ is a projective line or a point.
\end{itemize}
Suppose $-K_{X_r}$ is big.
If $C$ is a projective line, then $X$ is a Mori dream space.
If $C$ is a point, then $X$ is a Mori dream space or a projective cone of an elliptic curve.
\end{prop}

\begin{proof}
If $C$ is a projective line, then $X$ is klt by  Lemma \ref{lem:fiber-klt}.
In pariticular $X$ is $\Q$-factorial,
and $X$ is a Mori dream space by Lemma \ref{lem:anti-big-mds} and Lemma \ref{lem:mmp-anti-big}.

If $C$ is a point, then $X$ is a projective cone of an elliptic curve or rational surface with rational singularities by the last part in the proof of \cite[Theorem 5.1]{brou-gon}.
If $X$ has rational singularities, then $X$ is $\Q$-factorial by \cite[Theorem 4.6]{bades} and a Mori dream space by Lemma \ref{lem:anti-big-mds} and Lemma \ref{lem:mmp-anti-big}.

\end{proof}

\section{poof of the main theorem}

\begin{proof}[Proof of Theorem \ref{thm: main}]
Let $f \colon X \longrightarrow X$ be an int-amplified endomorphism of normal projective surface.
By Lemma \ref{lem:endom-and-sing}, $X$ is $\Q$-Gorenstein lc.
By \cite[Theorem 2.3.6]{kol-kov}, we can run a MMP for $X$.
By \cite[Theorem 4.6]{meng-zhang2} and Lemma \ref{equiv-contr}, 
if we replace $f$ by a suitable power, every $K_{X}$-negative extremal ray contraction is $f$-equivariant and
the induced morphism on the target is also int-amplified (Lemma \ref{property-of-intamp}).
Therefore, we can repeat this process and get 
\begin{align*}
X=X_{1} \longrightarrow \cdots \longrightarrow X_{r} \longrightarrow C 
\end{align*}
where
\begin{itemize}
\item $p_{i} \colon X_{i} \longrightarrow X_{i+1}$ is the divisorial contraction of a $K_{X_{i}}$-negative extremal ray for $i=1,\dots ,r-1$;
\item $X_{r} \longrightarrow C$ is identity or the Fano contraction of a $K_{X_{r}}$-negative extremal ray.
\end{itemize}
By replacing $f$ by its iterate, we assume $f$ induces int-amplified endomorphisms $f_{i}$ on $X_{i}$.

{\bf (1)} When $K_{X}$ is pseudo-effective, then by Lemma \ref{lem:eff-anti-canonical}, $K_{X}\equiv0$.
By \cite[Theorem 1.2]{gong}, $K_{X}\sim_{\Q}0$ and in particular, $r=1$ and $C=X$.
If $X$ is klt, then $X$ is a Q-abelian variety  by \cite[Theorem 5.2]{meng}.  

{\bf (2)} When $K_{X}$ is not pseudo-effective, then the out put of MMP must be a Fano contraction (cf. \cite[Corollary 1.1.7]{bchm}). 
Note that $p_{i}(\Exc (p_{i}))$ is a non-empty finite set totally invariant under $f_{i+1}$.

(a) If $C$ is an elliptic curve, by Proposition \ref{end-mfs-sing}(2), $f_{r}$ admits no totally invariant finite set.
Therefore, $r=1$ and by Proposition \ref{end-mfs-sing}(2) again, $X=X_{1}$ is smooth.

(b) If $C\simeq \P^{1}$ and $\kappa(-K_{X_{r}})=0$, then $r=1$ and $X$ is klt by Lemma \ref{lem:fiber-klt}.
By Proposition \ref{qetale-cover}, we get a desired quasi-\'etale cover as in the statement.

(c) If $C\simeq \P^{1}$ and $\kappa(-K_{X_{r}})=1$,
then $r=1$ and $X=X_{1}$ is klt by Lemma \ref{lem:fiber-klt}.
By Proposition \ref{qetale-cover2}, we get a desired quasi-\'etale cover as in the statement.

(d) If $C\simeq \P^{1}$ and $\kappa(-K_{X_{r}})=2$, then $X$ is klt by Lemma \ref{lem:fiber-klt} and hence $X$ is a Mori dream space by Proposition \ref{mmp-mds}

(e) If $C$ is a point, then $X_{r}$ has Picard number one and $-K_{X_{r}}$ is ample.
By Proposition \ref{mmp-mds}, $X$ is a Mori dream space or a projective cone of an elliptic curve.
\end{proof}

\section{Examples}

\begin{prop}
The cases Theorem \ref{thm: main} (\ref{k=0case}) (\ref{k=1case}) occur. 
\end{prop}

Let $E$ be an elliptic curve. We write $[m] \colon E \longrightarrow E$ the multiplication by $m$ map for every integer $m$.
Take an invertible $ \mathcal{O}_{E}$-module $ \mathcal{L}$ with $\deg \mathcal{L}=0$.
Consider  the projective bundle $p \colon Y=\P( \mathcal{O}_{E}\oplus \mathcal{L}) \longrightarrow E$.

\begin{lem}\label{lem:fixing-isoms}\ 
\begin{enumerate}
\item \label{inv}
For any isomorphism $\varphi \colon [-1]^{*} \mathcal{L} \longrightarrow \mathcal{L}^{-1}$, we have
\[
\xymatrix{
[-1]^{*}([-1]^{*} \mathcal{L}) \ar[d]_{[-1]^{*}\varphi} \ar[r] & ([-1]\circ [-1])^{*} \mathcal{L}  \ar[r] & \mathcal{L} \\
[-1]^{*}( \mathcal{L}^{-1}) \ar[r] & ([-1]^{*} \mathcal{L})^{-1} \ar[r]_{\varphi^{\vee}} & ( \mathcal{L}^{-1})^{-1} \ar[u]
}
\]
commutative, where unlabeled arrows are canonical isomorphisms.
\item \label{isoms}
Let $n>1$ an integer.
For every isomorphism $\varphi \colon [-1]^{*} \mathcal{L} \longrightarrow \mathcal{L}^{-1}$, 
there exists an isomorphism $\psi \colon [n]^{*} \mathcal{L} \longrightarrow \mathcal{L}^{n}$ such that
the following diagram is commutative:
\[
\xymatrix{
[-1]^{*}([n]^{*} \mathcal{L}) \ar[r]^{[-1]^{*}\psi} \ar[d]  &  [-1]^{*}( \mathcal{L}^{n}) \ar[r]  &  ([-1]^{*} \mathcal{L})^{n} \ar[r]^{\varphi^{\otimes n}}  &  (\mathcal{L}^{-1})^{n} \ar[d]\\
[n]^{*}([-1]^{*} \mathcal{L}) \ar[r]_{[n]^{*}\varphi} & [n]^{*}( \mathcal{L}^{-1}) \ar[r] & ([n]^{*} \mathcal{L})^{-1} \ar[r]_{\psi^{\vee}} & ( \mathcal{L}^{n})^{-1}
}
\]
where unlabeled arrows are canonical isomorphisms.
\end{enumerate}
\end{lem}
\begin{proof}
We may assume $ \mathcal{L}= \mathcal{O}_{E}(x-0)$ where $0\in E$ is the identity and $x\in E$ is a closed point.
Take any non-zero rational functions $f, g$ on $E$ so that
\begin{align*}
[-1]^{*}(x-0)&=-(x-0)+\dv f\\
 [n]^{*}(x-0)&=n(x-0)+\dv g.
\end{align*}

(\ref{inv})
We can reduce to prove that $([-1]^{*}f)/f=1$.
This function is constant by definition.
Take a two torsion point $z\in E\setminus \{0,x, y\}$, where $y\in E$ is the inverse element of $x$.
Then $(([-1]^{*}f)/f)(z)=f(z)/f(z)=1$.

(\ref{isoms})
We can reduce to find $g$ such that
\[
\frac{[n]^{*}f}{f^{n}g[-1]^{*}g}=1.
\]
The left hand side is constant, say $a$, by the definition of $f$ and $g$.
Replace $g$ by $ \sqrt{a}g$ and we get the desired one.

\end{proof}

Let $n>1$ be an integer.
Fix two isomorphisms $\varphi \colon [-1]^{*} \mathcal{L} \longrightarrow \mathcal{L}^{-1}$, $\psi \colon [n]^{*} \mathcal{L} \longrightarrow \mathcal{L}^{n}$
as in Lemma \ref{lem:fixing-isoms} (\ref{isoms}).

Consider the following diagram:
\[
\xymatrix{
Y \ar@/^20pt/[rrr]^{F} \ar[rd]_{p} \ar[r]_(.3){\alpha}  &  \P( \mathcal{O}_{E} \oplus \mathcal{L}^{n}) \ar[d] \ar[r]^(.45){\Psi}  &   \P( \mathcal{O}_{E} \oplus [n]^{*} \mathcal{L}) \ar[dl]^{[n]^{*}p} \ar[r]_(.7){\beta} & Y \ar[d]^{p}\\
 & E \ar[rr]_{[n]} & & E
}
\]
where $[n]^{*}p$ is the base change of $p$ by $[n]$, $\beta$ is the projection, $\Psi$ is the isomorphism over $E$ induced by $\psi$, and
$\alpha$ is the morphism over $E$ defined by the canonical inclusion $\mathcal{O}_{E}\oplus \mathcal{L}^{n} \longrightarrow \Sym^{n} ( \mathcal{O}_{E}\oplus \mathcal{L})$.
Define $F \colon Y \longrightarrow Y$ to be the composite $F=\beta \circ \Psi \circ \alpha$.
Note that $F$ is an int-amplified endomorphism.

Similarly, consider the following diagram:
\[
\xymatrix{
Y \ar@/^20pt/[rrr]^{\sigma} \ar[rd]_{p} \ar[r]_(.3){ \iota}  &  \P( \mathcal{O}_{E} \oplus \mathcal{L}^{-1}) \ar[d] \ar[r]^(.45){\Phi}  &   \P( \mathcal{O}_{E} \oplus [-1]^{*} \mathcal{L}) \ar[dl]^{[-1]^{*}p} \ar[r]_(.7){\gamma} & Y \ar[d]^{p}\\
 & E \ar[rr]_{[-1]} & & E
}
\]
where $[-1]^{*}p$ is the base change of $p$ by $[-1]$, $\gamma$ is the projection, $\Phi$ is the isomorphism over $E$ induced by $\varphi$, and
$ \iota$ is the isomorphism over $E$ induced by $ \mathcal{O}_{E}\oplus \mathcal{L} \simeq \mathcal{L} \oplus \mathcal{O}_{E} \simeq ( \mathcal{O}_{E}\oplus \mathcal{L}^{-1}) {\otimes} \mathcal{L}$.
Define $\sigma \colon Y \longrightarrow Y$ to be the composite $\sigma=\gamma \circ \Phi \circ \iota$.
Then, by Lemma \ref{lem:fixing-isoms} (\ref{inv}), we get $\sigma \circ \sigma=\id$.
By Lemma \ref{lem:fixing-isoms} (\ref{isoms}), we get $F\circ \sigma =\sigma \circ F$.

Let $X:=Y/\langle \sigma \rangle$ be the quotient of $Y$ by the involution $\sigma$. 
Then $X$ is a projective klt surface and  
we get the following commutative diagram:
\[
\xymatrix{
Y \ar[r]^{h} \ar[d]_{p} & X \ar[d]^{\pi} \\
E \ar[r] & E/\langle [-1] \rangle \simeq \P^{1}
}
\]
where the horizontal arrows are quotient morphisms and $\pi$ is the induced morphism by $p$.
Note that $h$ is quasi-\'etale since the set of fixed points of $\sigma$ is finite.
Since $F\circ \sigma =\sigma \circ F$, $F$ descends to an int-amplified endomorphism $f \colon X \longrightarrow X$.
Also, $[n] \colon E \longrightarrow E$ induces an endomorphism $g \colon \P^{1} \longrightarrow \P^{1}$ 
and the above diagram is equivariant under these endomorphisms.

We have $h^{*}K_{X}\sim K_{Y}$ because $h$ is quasi-\'etale.
Therefore, $\pi$ is a $K_{X}$-negative extremal ray contraction and $\kappa(-K_{X})=\kappa(-K_{Y})$. 
Moreover, $\kappa(-K_{Y})=0$ if $ \mathcal{L}$ is non-torsion in $\Pic^{0}(E)$ and $\kappa(-K_{Y})=1$ if $ \mathcal{L}$ is torsion.
The morphism $f \colon X \longrightarrow X$ is an example of the case Theorem \ref{thm: main} (\ref{k=0case}) or (\ref{k=1case})
depending on whether $ \mathcal{L}$ is non-torsion or not.

\begin{ack}
The authors would like to thank Sho Ejiri, Makoto Enokizono, Takeru Fukuoka, and Kenta Hashizume for answering their questions.
The authors are supported by  the Program for Leading Graduate Schools, MEXT, Japan.
The first author is supported by JSPS Research Fellowship for Young Scientists and KAKENHI Grant Number 18J11260.
\end{ack}

\end{document}